\documentclass[12pt]{article}
\usepackage{amsmath ,amsthm,amssymb,amstext}
\usepackage{amsfonts}
\usepackage{hyperref}
\usepackage[latin1]{inputenc}
\usepackage[T1]{fontenc}

\newcommand\nbOne{{\mathchoice {\rm 1\mskip-4mu l} {\rm 1\mskip-4mu l}
{\rm 1\mskip-4.5mu l} {\rm 1\mskip-5mu l}}}

\newcommand\R{\mathcal{M}_{p}(\mathbb{R}_{+}^{*})}

\newcommand\Rp{\mathbb{R}_{+}}

\newcommand\Rpe{\mathbb{R}_{+}^{*}}

\newcommand\gd{\rangle}

\newcommand\pt{\langle}

\newtheorem{theo}{Theorem}
\newtheorem{prop}[theo]{Proposition}

\newtheorem{cor}[theo]{Corollary}
\newtheorem{lm}{Lemma}

\newtheorem{rk}{Remark}

\newtheorem{df}{Definition}

\newenvironment{lemma1.5}{\textbf{Lemma 1.5}\itshape}

\newenvironment{mat}{\textbf{Malthusian Hypotheses.}\itshape}

\title{Self-similar branching Markov chains}
\author{Nathalie Krell}

\begin{document}
\maketitle

\begin{center}
\it{Laboratoire de Probabilités et Modèles Aléatoires,\\
 Universit\'e Pierre et Marie Curie,\\ 175 rue du Chevaleret,
75013 Paris, France.}
\end{center}

\vspace*{0.8cm}

\begin{abstract}
The main purpose of this work is to  study  self-similar branching Markov chains.
First we will construct  such a process. Then we will establish certain Limit Theorems using the theory of self-similar Markov processes.

\end{abstract}


\bigskip
\noindent{\bf A.M.S. Classification. }   60J80, 60G18, 60F25, 60J27.

\bigskip
\noindent{\bf e-mail. } krell@ccr.jussieu.fr

\vspace*{0.8cm}

\section{Introduction.}

 This work is a contribution to the study of a special type of branching Markov chains. We will construct a continuous time branching chain
 $\mathbf{X}$ which has a self-similar property and which takes its values in the space of finite point measures of $\mathbb{R}_{+}^{*}$.
  This type of process is a generalization of a self-similar fragmentation (see \cite{ber}), which may apply to cases where the size models
  non additive quantities as e.g. surface
energy in aerosols. We will focus on the case where the index of
self-similarity $\alpha$ is non-negative, which means that the
bigger individuals will reproduce faster than the smaller ones.
There is no loss of generality by considering this model, as the map
$x\rightarrow x^{-1}$ on atoms in $\Rpe$ transforms a self-similar
process with index  $\alpha$ into another one with index $-\alpha$
(and preserves the Markov property).

In this article we choose to construct  the process by bare  hand.
We extend the method used in \cite{ber} to deal with more general
processes where we allow  an individual to have a mass bigger than
that of its parent. We will explain in the sequel, which
difficulties this new set-up entails.
 There exists closely related  articles about branching
processes, like among others \cite{kyp1999}, \cite{kyp2000} from
Kyprianou  and \cite{cha0}, \cite{cha} from Chauvin. However notice
that the time of splitting of the process depends on the size of the
atoms of the process.

More precisely  we will first introduce a branching Markov chains as
a marked tree and we will obtain  a process, indexed by generations
(it is simply a random mark on the tree of generation, see Section
\ref{sectionarbredegeneration}). Thanks to a martingale which is
associated to the latter and the theory of random stopping lines on
a tree of generation, we will define the process indexed by time.
After having constructed the process, we will study the evolution of
the randomly chosen branch of the chain, from which we shall deduce
some Limit
 Theorems, relying on the theory of self-similar Markov processes.
In an appendix we will  consider the intrinsic process and give some
properties in the spirit of the article of Jagers \cite{jag}. By the
way we will show properties about the earlier martingale.

\section{The marked tree.}\label{sectionarbredegeneration}
In this part we will introduce a branching Markov chain as a marked
tree, which gives a genealogic description of the process that we
will construct. This terminology comes from Neveu in \cite{nev} even
if here the marked tree we consider is slightly different. First we
introduce some notations and definitions.

A finite point measure on $\Rpe$ is a finite sum of Dirac point
masses $\mathbf{s}=\sum_{i=1}^{n}\delta_{ s_{i}}$, where the $s_{i}$
are called the atoms of $\mathbf{s}$ and $n\geq 0$ is an arbitrary
integer. We shall often write $\sharp \mathbf{s}=n=\mathbf{s}(\Rpe)$
for the number of atoms of $\mathbf{s}$, and $\R$ for the space of
finite point measures on $\Rpe$. We also define for
$f:\Rp^{*}\rightarrow \mathbb{R}$ measurable function and
$\mathbf{s}\in\R$\[\langle f,\mathbf{s}\rangle:=\sum_{i=1}^{\sharp
s} f(s_{i}),\]
 by taking the sum over the atoms of $\mathbf{s}$ repeated according to their  multiplicity and  we will sometimes use the slight abuse of notation
 \[\langle f(x),\mathbf{s}\rangle :=\sum_{i=1}^{\sharp s} f(s_{i})\]
when $f$ is defined as a function depending on the variable $x$.
 We endow the space $\R$ with the topology of weak convergence, which means that $\mathbf{s}_{n}$ converge to $\mathbf{s}$ if and only if   $\langle f,\mathbf{s}_{n}\rangle $ converge to $\langle f,\mathbf{s}\rangle$ for all continuous bounded functions $f$.

Let $\alpha\geq 0$ be an index of self-similarity and $\nu$ be some
\textbf{probability measure} on $\R$. The aim of this
 work is to construct a  branching Markov chain $\mathbf{X}=((\sum_{i=1}^{\sharp \mathbf{X}(t)} \delta_{X_{i}(t)})_{t\geq
0})$  with values in  $\R$, which is self-similar with   index
$\alpha$ and has reproduction law $\nu$. The index of
self-similarity will play a part in the rate at which an individual
will reproduce and the reproduction law $\nu$ will specify the
distribution of the offspring.  We  stress that our setting includes
the case when
\begin{eqnarray}\nu(\exists i\ : \  s_{i}>1)>0,\label{fargmentplusgrand}\end{eqnarray}  which means that with a positive probability the size of
 a daughter can exceed that of her mother.

 To do that, exactly as described in Chapter 1
section 1.2.1 of \cite{ber}, we will construct a marked tree.

 We consider the Ulam Harris
labelling system
$$\mathcal{U} := \cup_{n=0}^{\infty}\mathbb{N}^{n},$$  with the
notation $\mathbb{N}=\{ 1, 2,...\}$ and $\mathbb{N}^{0} =
\{\emptyset\}$. In the sequel the elements of $\mathcal{U}$ are
called nodes (or sometimes also individuals) and the distinguished
node $\emptyset$ the root.  For each $u = (u_{1}, . . . , u_{n})\in
\mathcal{U}$, we call $n$ the \textit{generation} of $u$ and write
$|u| = n$, with the obvious convention $|\emptyset| = 0$. When $n
\geq 0$, $u = (u_{1}, . . . , u_{n})\in\mathbb{N}^{n}$ and
$i\in\mathbb{N}$,  we write $ui = (u_{1}, . . . , u_{n}, i)\in
\mathbb{N}^{n+1}$ for the $i$-th child of $u$. We also define for $u
= (u_{1}, . . . , u_{n})$ with $n\geq 2$,
$$mu=(u_{1}, . . . , u_{n-1})$$the \textit{mother} of $u$, $mu=\emptyset$ if $u\in \mathbb{N}$. If  $v=m^{n}u$ for some $n\geq 0$ we write  $v\preceq u$ and say that $u$
\textit{stems} from $v$. Additionally for $M$ a set of
$\mathcal{U}$, $M \preceq v$ means that $u\preceq v$ for some $u\in
M$. Generally we write $M\preceq L$ if all $x\in L$ stem from $M$.

Here it will be convenient to identify the point measure
$\mathbf{s}$ with the infinite sequence $(s_{1}, ..., s_{n},0,...)$
obtained by aggregation of infinitely many 0's to the finite
sequence of the atoms of $\mathbf{s}$.

In particular we say that a random infinite sequence
$(\xi_{i},i\in\mathbb{N})$ has the law $\nu$, if there is a (random)
index $n$ such that $\xi_{i}=0 \Leftrightarrow i>n$ and the finite
point measure $\sum_{i=1}^{n}\delta_{\xi_{i}}$ has the law $\nu$.
\begin{df}\label{def1}
Let two independent families of i.i.d. variables be indexed by the
nodes of the tree, $(\overline{\xi}_{u}, u\in\mathcal{U})$ and
$(\mathbf{e}_{u}, u \in\mathcal{U})$, where for each
$u\in\mathcal{U}$
$\overline{\xi}_{u}=(\widetilde{\xi}_{ui})_{i\in\mathbb{N}}$ is
distributed according to the law $\nu$, and
$(\mathbf{e}_{ui})_{i\in\mathbb{N}}$ is a sequence of i.i.d.
exponential variables with parameter 1. We define recursively for
some fixed $x>0$
$$\xi_{\emptyset}: =x,\ \ \ \  a_{\emptyset} := 0 ,\
\ \ \ \  \zeta_{\emptyset} := x^{-\alpha}\mathbf{e}_{\emptyset},$$
and for $u\in\mathcal{U}$  and $i\in\mathbb{N}$:
$$\xi_{ui} :=
\widetilde{\xi}_{ui}\xi_{u},\ \ \ \  a_{ui}: = a_{u} + \zeta_{u} ,\
\ \ \ \  \zeta_{ui} := \xi_{ui}^{-\alpha}\mathbf{e}_{ ui}.$$
To each node $u$  of the  tree $\mathcal{U}$, we associate  the
mark $(\xi_{u}, a_{u}, \zeta_{u})$ where $\xi_{u}$ is the size,
$a_{u}$ the birth-time and $\zeta_{u}$ the lifetime of the
individual with label $u$. We call $$T_{x}=((\xi_{u}, a_{u}, \zeta_{u})_{u\in\mathcal{U}})$$ a marked tree with root of size $x$, and the law associated is denoted by $\mathbb{P}_{x}$. Let $\Bar{\Omega}$ be the set of all the possible marked trees.

\end{df}

The size of the individuals $(\xi_{u}, u\in\mathcal{U})$ defines a
 multiplicative cascade (see the references in Section 3 of
\cite{ber2006}). However the latter is not sufficient to construct
the process $\mathbf{X}$, in fact we also need the information given
by $( (a_{u}, \zeta_{u}), u\in\mathcal{U})$.

Another useful concept is that of  \textit{ line}. A subset
$L\subset \mathcal{U}$ is a  line if for every $u,v\in L$, $u\preceq
v\Rightarrow u =v$. The \textit{pre-L-sigma algebra} is $$
\mathcal{H}_{L}:= \sigma (\widetilde{\xi}_{u},\mathbf{e}_{ u};
\exists l\in L :   u\preceq l).$$ A random set of individuals
$$\mathcal{J}:\Bar{\Omega}\rightarrow \mathcal{P}(\mathcal{U})$$ is
\textit{optional} if $\{\mathcal{J}\preceq L\}\in \mathcal{H}_{L}$
for all line $L \subset \mathcal{U}$, where $
\mathcal{P}(\mathcal{U})$ is the power set of  $\mathcal{U}$ . An
\textit{optional line} is a random line which is optional. For any
optional set $\mathcal{J}$ we define  the pre-$\mathcal{J}$-algebra
by:
$$A\in\mathcal{H}_{\mathcal{J}} \Leftrightarrow \forall L\ \ \hbox{line}\  \ \subset \mathcal{U}: A \cap \{\mathcal{J} \preceq L\}\in \mathcal{H}_{L}.$$

The first result is:

\begin{lm}\label{lemmearbre}
The marked tree constructed in Definition \ref{def1} satisfies the
strong Markov branching property:
 for $\mathcal{J}$  an optional  line and $
\varphi_{u}:\Bar{\Omega}\rightarrow [0,1],\ u\in \mathcal{U}$,  measurable functions, we get that,
$$ \left.\mathbb{E}_{1}\left(\prod_{u\in \mathcal{J}}\varphi_{u}\circ  T^{\xi_{u}}
\right|\mathcal{H}_{\mathcal{J}}\right)=\prod_{u\in \mathcal{J}}
\mathbb{E}_{\xi_{u}}(\varphi_{u}),$$ where $T^{\xi_{u}}$ is  the
marked tree extracted  from $T_{1}$ at the node $(\xi_{u}, a_{u},
\zeta_{u})$. More precisely
$$T^{\xi_{u}}=((\xi_{uv}, a_{uv}-a_{u}, \zeta_{uv})_{v\in\mathcal{U}}).$$
\end{lm}

\begin{proof}

Thanks to the i.i.d properties of the random variables
$(\widetilde{\xi}_{u}, u\in\mathcal{U})$ and $(\mathbf{e}_{u}, u
\in\mathcal{U})$, the Markov property for lines is of course easily
checked. In order to get the result for a more general optional
line,  we use  Theorem 4.14 of \cite{jag}. Indeed, the tree we have
constructed is a special case of the tree constructed by Jagers in
\cite{jag}. In our case the Jagers's notation $\rho_{u}$, $\tau_{u}$
and $\sigma_{u}$ are such that   the type $\rho_{u}$ of
$u\in\mathcal{U}$, is the mass of $u$: $\xi_{u}$, the birth time
$\sigma_{u}$ is $a_{u}$ and $\tau_{u}$ is here equal to $\zeta_{mu}$
(because the mother dies when she gives
 birth  to her daughters). We
 notice that all the sisters have the same birth time, which means
 that for all $u\in\mathcal{U}$ and all $i\in\mathbb{N}$, we have
 that $\tau_{ui}$ is here equal to $\zeta_{u}$.
\end{proof}

\section{Malthusian hypotheses and the
intrinsic martingale.}

We  introduce  some notations  to formulate the fundamental
assumptions of this work:
$$\underline{p} := \inf \left\{p\in\mathbb{R} :
\int_{\R}\pt x^{p},\mathbf{s}\gd\nu(d\mathbf{s})<\infty\right\},$$
and
$$p_{\infty} := \inf\left\{p >\underline{p}  :
\int_{\R}\pt x^{p},\mathbf{s}\gd \nu(d\mathbf{s})=\infty\right\}$$
(with the convention $\inf \emptyset = \infty$) and then for every
$p\in (\underline{p},p_{\infty})$:
$$\kappa(p):=\int_{\R}\left(1-\pt x^{p},\mathbf{s}\gd\right)
\nu(d\mathbf{s}).$$
  Note that $\kappa$ is a continuous and concave function (but not necessarily a strictly increasing function) on
  $(\underline{p},p_{\infty})$, as $p \rightarrow\int_{\R} \pt x^{p},\mathbf{s}\gd\nu(d\mathbf{s})$ is a convex application.
By concavity, the equation $\kappa (p)=0$ has at most two solutions
on $(\underline{p},p_{\infty})$. When a solution exists, we denote
by $p_{0}:=\inf\{p\in(\underline{p},p_{\infty}): \kappa(p) =0\}$ the
smallest, and call $p_{0} $ the  Malthusian exponent.

We now make the fundamental:

\medskip

\begin{mat} We suppose that the Malthusian exponent $p_{0}$ exists, that  $p_{0}>0$, and that
 \begin{equation}  \kappa(p)>0 \ \ \text{for some}\ \  p>p_{0}.\label{hypomalt}\end{equation}
 Furthermore we suppose that
the integral
\begin{equation} \int_{\R}\left( \pt
x^{p_{0}},\mathbf{s}\gd \right)^{p}\nu(d\mathbf{s})
\label{hypomalt2}\end{equation}
 is finite
for some $p > 1$.
\end{mat}

\medskip

 \textbf{Throughout the rest of this article,
  these hypotheses will
always be taken for granted.}

Note that ~(\ref{hypomalt}) always holds when $\nu (s_{i}\leq 1 \
\text{for all}\ i)=1$ (fragmentation case). We stress that $\kappa $
may not be strictly increasing, and may not be negative when $p$ is
sufficiently large (see Subsection \ref{subsect61} for a consequence
of this fact.)

We will give one example based on the  Dirichlet process (see the
book Kingman \cite{kin}).  Fix $n\geq 2$,
$(\upsilon_{1},...,\upsilon_{n})$ $n$ positive real numbers and
$\upsilon=\sum_{i=1}^{n}\upsilon_{i}$. We define the simplex
$\Delta_{n}$ by
\[\Delta_{n}:=\left\{(p_{1},p_{2},...,p_{n})\in\mathbb{R}_{+}^{n},\ \sum_{j=1}^{n}p_{i}=1\right\}.\]
The Dirichlet distribution  of parameter
$(\upsilon_{1},...,\upsilon_{n})$ over the simplex $\Delta_{n}$ has
the density (with respect to the $(n-1)$-dimensional Lebesgue
measure on $\Delta_{n}$):
\[f(p_{1},...,p_{n})=\frac{\Gamma( \upsilon)}{\Gamma (\upsilon_{1})...\Gamma (\upsilon_{n})}p_{1}^{\upsilon_{1}-1}...p_{n}^{\upsilon_{n}-1}.\]

 Let
$a:=\upsilon
(\upsilon+1)/(\sum_{i=1}^{n}\upsilon_{i}(\upsilon_{i}+1))$. Note
that $a$ is strictly larger than 1. Let the reproduction measure be
the law of $(aX_{1},...,aX_{n})$, where $(X_{1},...,X_{n})$ is a
random vector with Dirichlet distribution of parameter
$(\upsilon_{1},...,\upsilon_{n})$. Therefore
$$\kappa (p)=a^{p}\frac{\Gamma (\upsilon )}{\Gamma (\upsilon
+p)}\sum_{i=1}^{n}\frac{\Gamma (p+\upsilon_{i})}{\Gamma
(\upsilon_{i}) },$$ $\underline{p}=-\upsilon$, $p_{0}=1$ and the
Malthusian hypotheses are verified.

In this article we will call \textit{extinction} the event that for
some  $n\in\mathbb{N}$, all nodes $u$ at the $n$-th generation have
zero size, and \textit{non-extinction} the complementary event. We
see that  the probability of extinction is always strictly positive
whenever $\nu (s_{1}=0)>0$, and equals zero if and only if
$\nu(s_{1} = 0) = 0$ (as we have suppose~(\ref{hypomalt2}); see p.28
\cite{ber}).

After these definitions, we introduce a fundamental martingale
associated to $(\xi_{u}, u\in\mathcal{U})$.

\begin{theo}\label{theoremdecvmartingale} The process $$M_{n} := \sum_{|u|=n}\xi_{u}^{p_{0}},\ \ \ n\in\mathbb{N}$$
is a martingale in the filtration $(\mathcal{H}_{L_{n}})$, with
$L_{n}$ the line associated to the $n$-th generation (i.e.
$L_{n}:=\{ u \in \mathcal{U}:\ |u|=n\}$). This martingale is bounded
in $\mathrm{L}^{p}(\mathbb{P})$ for some $p
> 1$, and in particular is uniformly integrable.

Moreover, conditionally on non-extinction
 the terminal value
$M_{\infty}$ is strictly positive a.s.
\end{theo}

\begin{rk}
As $\kappa$ is concave the equation $\kappa (p)=0$ may have a second
root $p_{+}:=\inf\{p>p_{0},\ \kappa(p)=0\}$). This second root is
less interesting:  even though
$$M_{n}^{+} := \sum_{|u|=n}\xi_{u}^{p_{+}},\ \ \ n\in\mathbb{N},$$ is also a martingale, it is easy to check that  for all $p>1$ the $p$-variation of
$M_{n}^{+}$ is infinite, i.e. $\mathbb{E}\left(
\sum_{n=0}^{\infty}|M_{n+1}-M_{n}|^{p}\right)=\infty$).

We can notice that for all $p\in (p_{0},p_{+})$
$(M_{n}^{(p)})_{n\in\mathbb{N}}:=(\sum_{|u|=n}\xi_{u}^{p})_{n\in\mathbb{N}}$
is a supermartingale.

The assumption (\ref{hypomalt2}) means actually that
$\mathbb{E}(M_{1}^{p})<\infty$.
\end{rk}

 \begin{proof}
 $\bullet$ We
will use the fact that the empirical measure of the logarithm of
the sizes of  fragments
\begin{equation}Z^{(n)}:=\sum_{|u|=n}\delta_{\log \xi_{u}}\label{zn}\end{equation} can be
viewed as a branching random walk (see the article of Biggins
\cite{big})  and use Theorem 1 of \cite{big}. In order to do that we
first introduce some notation:  for $ \theta > \underline{p}$, we
define
 $$m(\theta ):= \mathbb{E}\left(\int e^{\theta x
}Z^{(1)}(dx)\right)=\mathbb{E}\left(\sum_{|u|=1}\xi_{u}^{\theta
}\right)=1-\kappa (\theta)$$ and $$W^{(n)} (\theta ):=m (\theta
)^{-n} \int e^{\theta x }Z^{(n)}(dx)=(1-\kappa
(\theta))^{-n}\sum_{|u|=n}\xi_{u}^{\theta }. $$ We notice that
$M_{n}=W^{(n)} (p_{0})$. Therefore in order to apply Theorem 1 of
\cite{big} and to get the convergence almost surely and in $p$th
mean for some $p>1$, it is enough to show that
$$\mathbb{E}(W^{(1)} (p_{0})^{\gamma })<\infty $$ for some $\gamma
\in(1,2]$ and
$$m(pp_{0})/ |m (p_{0})|^{p}<1$$ for some $p\in (1, \gamma ]$.
The first condition is a consequence of the Malthusian assumption.
Moreover the second follows from the identities
$$m(pp_{0})/ |m (p_{0})|^{p}=(1-\kappa(pp_{0}))/ |1-\kappa
(p_{0})|^{p}=1-\kappa(pp_{0})$$ which, by the definition of $p_{0}$,
is smaller than 1 for $p>1$ well chosen.

$\bullet$ Finally, let us now check that $M_{\infty}> 0$ $a.s.$
conditionally on non-extinction. Define $q = \mathbb{P} (M_{\infty}
= 0)$, therefore as $\mathbb{E}(M_{\infty}) = 1$ we get that $q <
1$. Moreover, an application of the branching property yields
$$\mathbb{E}(q^{Z_{n}})= q ,$$ where $Z_{n}$ is the number of
individuals with positive size at the $n$-th generation. Notice that
$Z_{n}=\pt Z^{(n)},1\gd $. By the construction of the marked tree
and as $\nu$ is a probability measure: $(Z_{n}, n\in\mathbb{N})$ is
of course a Galton-Watson process and it follows that $q$ is its
probability of extinction. Since $M_{\infty}= 0$ conditionally on
the extinction, the two events coincide a.s.
\end{proof}

\section{Evolution of the process in continuous time.}\label{section4}

After having defined the process indexed by generation and having
shown that the martingale $M_{n}$ is $\mathrm{L}^{p}(\mathbb{P})$
bounded, we are now able to define properly the main objet of this
paper. In order to do this,  when an individual labelled by $u$ has
a positive size, $\xi_{u}>0$, let $I_{u}:=[a_{u},a_{u}+\zeta_{u})$
be the interval of times during which this individual  is alive.
Otherwise, i.e. when $\xi_{u}=0$, we decide that $I_{u}=\emptyset$.
With this definition, we set:

\begin{df}\label{lmdefduprocesus} We define the process $\mathbf{X}=(\mathbf{X}(t),t\geq 0)$ by  \begin{eqnarray}\mathbf{X}(t)=\sum_{u\in\mathcal{U}}\nbOne_{\{t\in
I_{u}\}}\delta_{\xi_{u}},t\geq
0.\label{defduprocessus}\end{eqnarray} In particular we have for
$f:\Rp\rightarrow \mathbb{R}$ measurable function  $$\pt f,
\mathbf{X}(t)\gd=\sum_{u\in\mathcal{U}} f(\xi_{u})\nbOne_{\{t\in
I_{u}\}} .$$
\end{df}

For every $x>0$, let $\mathbb{P}_{x}$ be the law of the process
$\mathbf{X}$ starting from a single individual with size $x$. And
for simplification, we denote   $\mathbb{P}$ for  $\mathbb{P}_{1}$,
and let  $(\mathcal{F}_{t})_{t\geq 0}$ be the natural filtration of
the process $(\mathbf{X}(t), t \geq 0)$. We use the notation
$(X_{1}(t),..., X_{\sharp \mathbf{X}(t)}(t))$ for the sequence of
atoms of $ \mathbf{X}(t)$. In the following we will show that this
sequence is almost surely  finite. Of course the set $(X_{1}(t),...,
X_{\sharp \mathbf{X}(t)}(t))$ is the same as the set $((\xi_{u});
t\in I_{u})$; but sometimes it will be clearer to use the notation
$(X_{i}(t))$.

We define for $u\in\mathbb{R}_{+}$: $$F(u):=\int_{\R}u^{\sharp
\mathbf{s}}\nu(d \mathbf{s}).$$  We notice that $F(u)$ is the
generating function of the Galton-Watson process $(Z_{n}, n\geq
0)=(\sharp \{ u\in\mathcal{U}: \ \xi_{u}>0 \ \text{and}\ \ |u|=n\},n
\geq 0)$. 

\textbf{From now on, we will suppose} that for every $\epsilon >0$
\begin{eqnarray}
\int_{1-\epsilon}^{1}\frac{du}{F(u)-u}=\infty.
\label{hyposurmesurededis}\end{eqnarray} Of course if
$F^{'}(1)=\mathbb{E}(Z_{1})<\infty$ this last assumption is
fulfilled. Therefore we get the first theorem about the continous
time process:

\begin{theo}\label{theodef} The
process $\mathbf{X}$ takes its values in the set $\R$. It is a
branching Markov chain, more precisely the conditional distribution
of $\mathbf{X}(t+r)$ given that $\mathbf{X}(r)=\mathbf{s}$ is the
same as that of the sum $\sum \mathbf{X}^{(i)}(t)$,  where for each
index $i$, $\mathbf{X}^{(i)}(t)$ is distributed as $\mathbf{X}(t)$
under $\mathbb{P}_{s_{i}}$ and the variables $\mathbf{X}^{(i)}(t)$
are independent.

The process $\mathbf{X}$ also has the scaling property, namely for
every $c>0$, the distribution of the rescaled process
$(c\mathbf{X}(c^{\alpha}t), t\geq 0)$ under $\mathbb{P}_{1}$  is
$\mathbb{P}_{c}$.

\end{theo}

In the fragmentation case, the fact that the size of the fragments
decreases with time entails that the process of the fragments of
size larger  than or equal to $\epsilon$ is Markovian, and which
leads easily to Theorem \ref{theodef}. This property is lost in the
present case.

\begin{proof}

$\bullet$ First we will check that for all $t\geq 0$,
$\mathbf{X}(t)$ is a (random) finite point measure.
 By Theorem~\ref{theoremdecvmartingale} and
the Doob's $ \mathrm{L}^{p}$-inequality we get that for some
$p>1$:
$$\sup_{n\in\mathbb{N}}M_{n}=\sup_{n\in\mathbb{N}}\sum_{|u|=n}\xi_{u}^{p_{0}}\in
\mathrm{L}^{p}( \mathbb{P}).$$ As a consequence:
$$\sup_{u\in\mathcal{U}}\xi_{u}^{p_{0}}\in
\mathrm{L}^{p}( \mathbb{P})$$ and then by the definition of the
process $\mathbf{X}$, writing $X_{1}(t),...$ for the (possibly
infinite) sequence of atoms of $\mathbf{X}(t)$
$$\sup_{i}\sup_{t\in\mathbb{R}_{+}}X_{i}(t)^{p_{0}}\in
\mathrm{L}^{p}( \mathbb{P}).$$ Recall that $p_{0}>0$ by assumption.
We fix some arbitrarily large  $m>0$. We now work conditionally on
the event that the size of all individuals is bounded by $m$, and we
will  show that the number of the individuals alive at time $t$ is
almost surely finite for all $t\geq 0$.

As we are conditioning on the event
$\{\sup_{u\in\mathcal{U}}\xi_{u}\leq m\}$,  by the construction of
the marked tree, we get that the life time of an individual  can
be stochastically  bounded from below by an exponential variable
of parameter $m^{\alpha}$. Therefore we can bound the number of
individuals present at time $t$ by the number of individuals of a
continuous time branching process denoted by $GW$ in which  each
individual lives for a random time  whose law is  exponential of
parameter
 $m^{\alpha}$  and the probability distribution of the offspring is the law of $ \sharp s\vee 1$ under $\nu$ (we have taken the supremum
  with 1 to ensure the absence of death).
For the Markov branching process $GW$, we are in the temporally
homogeneous case and, we notice that  $$\int_{\R}u^{(n_{
\mathbf{s}})\vee 1}\nu(d\mathbf{s})=(f(u)-u)\nu (n_{ \mathbf{s}}\neq
0)+u,$$ therefore as we have supposed (\ref{hyposurmesurededis}), we
can use Theorem 1 p.105 of the book of Athreya and Ney \cite{athney}
(proved in Theorem 9 p.107 of the book of Harris \cite{har}) and get
that we are in the non-explosive case for the $GW$. As the number of
the individuals is bounded by that of  $GW$ we get that the number
of individuals at time $t$ is a.s. finite.

Therefore  conditioning on the event
$\{\sup_{u\in\mathcal{U}}\xi_{u}\leq m\}$, we have that for all
$t\geq 0$, the number of individuals at time $t$ is a.s. finite,
i.e. $\mathbf{X}(t)$ is a finite point measure.

$\bullet$ Second we will show the Markov property. Fix
$r\in\mathbb{R}_{+}$. Let $\tau_{r}$ be equal to $\{
u\in\mathcal{U}:\ \ r\in I_{u}\}$. We notice that $\tau_{r}$ is an
optional line. In fact for all lines $L\subset \mathcal{U}$ we have
that  $$\{ \tau_{r}\preceq  L\} =\{  r <a_{u} +\zeta_{u} \ \ \forall
u\in L\}\in \mathcal{H}_{L}.$$ By definition, we have the identity
$$\sum_{i=1}^{\sharp\mathbf{X}(t+r)} \nbOne_{\{X_{j}(t+r)>0\}}
\delta_{X_{j}(t+r)}=\sum_{u\in\mathcal{U}}\nbOne_{\{t+r\in
I_{u}\}}\delta_{\xi_{u}}.$$ Let
$\mathbf{X}(r)=\sum_{i=1}^{n}\delta_{\xi_{v_{n}}}\in\R$ with
$n=\sharp\mathbf{X}(r)$ and  $(v_{1},...,v_{n})$ the nodes of
$\mathcal{U}$.
 Define for all $i\leq n$,
  $$\tilde{T}^{(i)}:=((\xi_{v_{i}u}, a_{v_{i}u}-a_{v_{i}}, \zeta_{v_{i}u}-\nbOne_{\{u=\emptyset\}}(r-a_{v_{i}}))_{u\in\mathcal{U}})
 =((\tilde{\xi}_{u}^{(i)}, \tilde{a}_{u}^{(i)}, \tilde{\zeta}_{u}^{(i)})_{u\in\mathcal{U}}), $$
$\tilde{I}_{u}^{(i)}:=[\tilde{a}_{u}^{(i)},
\tilde{a}_{u}^{(i)}+\tilde{\zeta}_{u}^{(i)}[$ and
$$\mathbf{X}^{(i)}(t)=\sum_{u\in\mathcal{U}}\nbOne_{\{t\in
\tilde{I}_{u}^{(i)}\}}\delta_{\tilde{\xi}_{u}^{(i)}}.$$  Then
$$\mathbf{X}(t+r)=\sum_{i=1}^{n}\mathbf{X}^{(i)}(t).$$

By the lack of memory of the exponential variable, we have that for
$u\in\mathcal{U}$, given $s\in I_{u}$ the law of the marked tree
$\tilde{T}^{(i)}$  is the same as  that of
$$T^{\xi_{v_{i}}}:=((\xi_{v_{i}u},
a_{v_{i}u}-a_{v_{i}},
\zeta_{v_{i}u})_{u\in\mathcal{U}}):=((\xi_{u}^{i}, a_{u}^{i},
\zeta_{u}^{i})_{u\in\mathcal{U}}) .$$
 Thus we have the equality in law: $$\sum_{u\in\mathcal{U} }\nbOne_{\{t\in
\tilde{I}_{u}^{(i)}\}}\delta_{\tilde{\xi}_{u}^{(i)}}\overset{(d)}{=}\sum_{u\in
\mathcal{U}}\nbOne_{\{t\in I_{u}^{i}\}}\delta_{\xi_{u}^{i}},$$ with
$I_{u}^{i}:=[a_{u}^{i},a_{u}^{i}+\zeta_{u}^{i}[$.

Let $\tau_{r}^{i}:=\{v_{i}u\in\mathcal{U}:\ \ r\in I_{u}^{i}\}$.
Moreover for all lines $L\in \mathcal{U}$ we have that  $$\{
\tau_{r}^{i}\preceq  L\} =\{ r<a_{v_{i}u} +\zeta_{v_{i}u} \ \
\forall v_{i}u\in L\}\in \mathcal{H}_{L}.$$ Therefore $\tau_{r}^{i}$
is an optional line and by applying Lemma~\ref{lemmearbre}  for the
optional  line $\tau_{s}^{i}$,  we have that the condition
distribution of the point measure
$$\sum_{u \in \mathcal{U}}\nbOne_{\{t+r\in
I_{u}^{i}\}}\delta_{\xi_{u}^{i}}$$ given $\mathcal{H}_{\tau_{r}}$ is
the law of
 $\mathbf{X}(t)$ under $\mathbb{P}_{x_{i}}$. We notice that $
 \mathcal{H}_{\tau_{s}}=\sigma (\tilde\xi_{u}, \mathbf{e}_{u}:
 a_{u}\leq s)$ is the same filtration as $ \mathcal{F}_{s}=\sigma
 (\mathbf{X}(s^{'}):\ s^{'} \leq s)$.
Therefore $(\mathbf{X}^{(1)}, \mathbf{X}^{(2)},...,
\mathbf{X}^{(n)})$ is a sequence of independent random processes,
where for each $i$ $\mathbf{X}^{(i)}(t)$ is distributed as
$\mathbf{X}(t)$ under $\mathbb{P}_{x_{i}}$. We then have proven the
Markovian property.

$\bullet$  The scaling property is an easy consequence of the
definition of the tree $T_{x}$.
\end{proof}

\begin{rk}For every measurable function $g : \mathbb{R}_{+}^{*}\rightarrow
\mathbb{R}_{+}^{*}$,  define a multiplicative functional  such that
for every $\mathbf{s}=\sum_{i=1}^{\sharp s} \delta_{s_{i}}\in\R $:
$$\phi_{g}(\mathbf{s}):=\exp  (- \pt g,\mathbf{s}\gd) = \exp  (-\sum_{i=1}^{\sharp s} g(s_{i}))       .$$ Then the generator $G$ of the Markov process $\mathbf{X}(t)$  fulfills for every  $\mathbf{y} =\sum_{i=1}^{\sharp\mathbf{y}} \delta_{y_{i}}\in\R$:
\begin{eqnarray}G\phi_{g}(\mathbf{y}) =
\sum y_{i}^{\alpha}e^{ -\sum_{j\neq i} g(y_{j})}\int_{ \R}(e^{ -\pt
g(xy_{i}),\mathbf{s}\gd }-e^{ -g(y_{i})})\nu(d\mathbf{s})
.\label{generator}\end{eqnarray}
\end{rk}

The intrinsic martingale $M_{n}$ is indexed by the generations; it
will  also be convenient to consider its analogue in continuous
time, i.e $$M(t):= \pt
x^{p_{0}},\mathbf{X}(t)\gd=\sum_{u\in\mathcal{U}}\nbOne_{\{t\in
I_{u}\}}\xi_{u}^{p_{0}} .$$
 It is
straightforward to check that $(M(t), t \geq 0)$ is again a
martingale in the natural filtration $(\mathcal{F}_{t})_{t\geq 0}$
of the process $(\mathbf{X}(t), t \geq 0)$; and more precisely, the
argument Proposition 1.5 in \cite{ber} gives:

\begin{cor}\label{convergencemartingale} The process $(M(t),t\geq 0)$ is a martingale, and more precisely  $$M(t) = \mathbb{E}(M_{\infty}| \mathcal{F}_{t}),$$  where $M_{\infty}$ is the terminal
value of the intrinsic martingale $(M_{n}, n \in \mathbb{N})$. In particular $M(t)$ converges in
$\mathrm{L}^{p}(\mathbb{P})$ to $M_{\infty}$ for some $p > 1$.
\end{cor}

\begin{proof}

 We will use the same argument as in the proof of Proposition 1.5 of \cite{ber}. Netherless, we have to deal here with the fact
  that  $\sup_{u\in\mathcal{ U}}\xi_{u}$ may be larger than 1.
 Therefore we will have to condition. We know that $M_{n}$ converges in $\mathrm{L}^{p}(\mathbb{P})$ to $M_{\infty}$ as $n$ tends to $\infty$, so $$\mathbb{E}(M_{\infty}| \mathcal{F}_{t}) =\lim_{n\rightarrow\infty}  \mathbb{E}(M_{n} |\mathcal{F}_{t}).$$ By Theorem~\ref{theoremdecvmartingale} as we have
$$\sup_{u\in\mathcal{U}}\xi_{u}^{p_{0}}\in\mathrm{L}^{p}(
\mathbb{P}),$$ we fix $m>0$. We now work  on the event
$B_{m}:=\{\sup_{u\in\mathcal{U}}\xi_{u}\leq m\}$.

By applying  the Markov property at time $t$ we easily get that
\begin{eqnarray}\mathbb{E}(M_{n} |\mathcal{F}_{t}) = \sum_{i=1}^{\sharp \mathbf{X}(t)} X_{i}^{p_{0}}(t)\nbOne_{\{\varrho (X_{i}(t))\leq n\}}+\sum_{|u|=n}\xi_{u}^{p_{0}}\nbOne_{\{a_{u}+\zeta_{u}<t\}}\label{equationdemartingale}\end{eqnarray}
 where $\varrho(\xi_{v})$ stands for the generation of the individual $v$ (i.e. $\varrho (\xi_{v}) = |v|$), and $a_{u} + \zeta_{u}$ is the instant when the individual corresponding to the node $u$ reproduces. We can rewrite the latter as
 $$a_{u}+\zeta_{u}=\xi_{m^{|u|}u}^{-\alpha}\mathbf{e}_{0}+\xi_{m^{|u|-1}u}^{-\alpha}\mathbf{e}_{1}+...+\xi_{u}^{-\alpha}\mathbf{e}_{|u|}$$ where $\mathbf{e}_{0}$,... is a sequence of independent exponential variables with parameter 1, which is also independent of $\xi_{u}$.
 We can remark that in the first term of  sum~(\ref{equationdemartingale}) we sum over the sizes of the individuals  which belong to the $n$-th generation and  are alive at time $t$, and in the second term we sum over those belonging to the $n$-th generation and  are dead at time $t$.

As $\alpha$ is nonnegative,  and as we are working on the event
$B_{m}$: $\xi_{m^{i}u}^{-\alpha}\geq m^{-\alpha}$ we have that for
each fixed node $u \in\mathcal{U}$ , $a_{u} + \zeta_{u}$ is bounded
from below by the sum of $|u|+1$ independent exponential variables
with parameter $m^{\alpha}$ which are independent of $\xi_{u}$. Thus
$$\lim_{n\rightarrow\infty}\mathbb{E}\left(\sum_{|u|=n}\xi_{u}^{p_{0}}\nbOne_{\{a_{u}+\zeta_{u}<t\}}\nbOne_{\{B_{m}\}}\right)
= 0 ,$$ and therefore  by (\ref{equationdemartingale})  on the event
$\{ B_{m}\}$, we get that for all $m> 0$: $\mathbb{E}(M_{\infty}|
\mathcal{F}_{t})\nbOne_{\{B_{m}\}} =M(t)\nbOne_{\{B_{m}\}}$, and
then by letting  $m$ tend to $\infty$ we get the result.
\end{proof}

\section{A randomly tagged leaf.}
We will here (as in \cite{ber}) define what a  tagged
individual is by using a tagged leaf.

We call \textit{leaf} of the tree $\mathcal{U}$ an infinite sequence
of integers $l = (u_{1}, . . .)$. For each $n$, $l^{n} := (u_{1}, .
. . , u_{n})$ is the  ancestor of $ l$  at the generation $n$. We
enrich the probabilistic structure by adding the information about a
so called tagged leaf, chosen at random as follows. Let
$\mathrm{H}_{n}$ be the space of bounded functionals $\Phi$ which
depend on the mark $M$ and  of the leaf $l$ up to the $n$-th first
generation, i.e. such that $\Phi(M, l) = \Phi(M^{'}, l^{'})$ if
$l^{n} = l^{n'}$ and $M(u) = M^{'}(u)$ whenever $|u| \leq n$. For
such functionals, we use the slightly abusing notation $\Phi(M, l) =
\Phi(M, l^{n})$. As in \cite{ber} for a pair $(M,\lambda )$ where $M
: \mathcal{U} \rightarrow [0, 1] \times \Rp\times \Rp$ is a random
mark on the tree and $\lambda$ is a random leaf of $\mathcal{U}$,
the joint distribution denoted by $ \mathbb{P}^{*}$ (and by $
\mathbb{P}_{x}^{*}$ if the size of the first mark is $x$ instead of
1) can be defined unambiguously   by
$$ \mathbb{E}^{*}(\Phi(M,\lambda))= \mathbb{E}\left(\sum_{|u|=n}\Phi(M,u)\xi_{u}^{p_{0}}\right),\ \ \ \ \Phi\in \mathrm{H}_{n}.$$
Moreover since the intrinsic martingale $(M_{n}, n \in
\mathbb{Z}_{+})$ is uniformly integrable (cf.
Theorem~\ref{theoremdecvmartingale}), the first marginal of $
\mathbb{P}^{*}$ is absolutely continuous with respect to the law of
the random mark $M$ under $ \mathbb{P}$, with density $M_{\infty}$.

Let $\lambda_{n}$ be  the node of the tagged leaf at the $n$-th
generation. We denote $\chi_{n}: =\xi_{\lambda_{n}}$ for the size of
the individual corresponding to the node $\lambda_{n}$  and $\chi(t)$
for the size of the tagged individual alive at time $t$, viz.
$$\chi(t):=\chi_{n}\ \ \ \ \hbox{if} \ \ a_{\lambda_{n}}\leq
t<a_{\lambda_{n}}+\zeta_{\lambda_{n}},$$ because in the case
considered $\sup_{n\in \mathbb{N}}a_{\lambda_{n}}=\infty$. We stress
that, in general the process $\chi (t)$ is not monotonic. However as
in \cite{ber}, Lemma 1.4 there
 becomes:

\begin{lm}\label{lemmeequaliteloikappa}
Let $k: \mathbb{R}_{+} \rightarrow \mathbb{R}_{+}$ be a measurable
function such that $k(0) = 0$. Then we have for every $n
\in\mathbb{N}$  $$ \mathbb{E}^{*}(k(\chi_{n})) =
\mathbb{E}\left(\sum_{|u|=n}\xi_{u}^{p_{0}}k(\xi_{u})\right) ,$$ and
for every $t \geq 0$
$$\mathbb{E}^{*}(k(\chi(t))) =
\mathbb{E}\left(\pt x^{p_{0}}k(x), X(t)\gd\right)
.$$
\end{lm}

Proposition 1.6  of \cite{ber} becomes:

\begin{prop}\label{propSn}
Under $ \mathbb{P}^{*}$, $$S_{n} := \ln \chi_{n} ,\ \ \  n \in
\mathbb{Z}_{+}$$is a random walk on $\mathbb{R}$ with step distribution
$$ \mathbb{P} (\ln \chi_{n} - \ln \chi_{n+1}\in  dy) =
\widetilde{\nu}(dy) ,$$
where the probability measure $\widetilde{\nu}$ is defined by
$$\int_{]0,\infty[}k(y) \widetilde{\nu}(dy)=\int_{ \R}\pt x^{p_{0}}k(\ln(x)), \mathbf{s}\gd\nu(d\mathbf{s}) .$$
Equivalently, the Laplace transform of the step distribution is
given by $$ \mathbb{E}^{*}(e^{pS_{1}})=
\mathbb{E}^{*}(\chi_{1}^{p})=1-\kappa(p+p_{0}),\ \ \ p\geq 0.$$ Moreover, conditionally
on $(\chi_{n}, n \in \mathbb{Z}_{+})$ the sequence of the
lifetimes $(\zeta_{\lambda_{0}},\zeta_{\lambda_{1}},...)$ along the
tagged leaf is a sequence of independent exponential variables
with respective parameters $\chi_{0}^{\alpha},\chi_{1}^{\alpha},...$
\end{prop}

We now see that we can use this proposition to obtain the
description of $\chi (t)$ using a Lamperti transformation. Let
$$\eta_{t}:=S\circ N_{t},\ \ \ \ t\geq 0,$$ with $N$ a Poisson
process with parameter 1 which is independent of the random walk
$S$; for probabilities and expectations related to $\eta$ we use the
notation $P$ and $E$. The process $ (\chi(t),t\geq 0)$ is Markovian
and enjoys a scaling property. More precisely under $
\mathbb{P}_{x}^{*}$ we get that
\begin{equation}\chi(t)\overset{(d)}{=}\exp(\eta_{\tau(tx^{-\alpha})}),\ \ \ t\geq 0,\label{lamperti}\end{equation}where
$\eta$ is the compound Poisson defined above and $\tau$ the
time-change defined implicitly by
\begin{equation}t=\int_{0}^{\tau(t)}\exp(\alpha\eta_{s})ds,\ \ \ t\geq
0.\label{lampertit}\end{equation}

\section{Asymptotic behaviors.}
\subsection{The convergence of the size of a tagged
individual.}\label{subsect61}

\noindent Let
 \[\kappa^{'}(p_{0})=-\int_{\R}\pt x^{p_{0}}\ln(x), \mathbf{s}\gd\nu (d\mathbf{s})\] denote the derivative of $\kappa$ at the Malthusian parameter $p_{0}$.

In this part we focus on the asymptotic behavior of the size of a
tagged individual. In this direction, the quantity
$\varpi_{t}=e^{\alpha \eta_{t}}$ plays an important role, as it
appears at the time  change of the Lamperti transformation (see
(\ref{lampertit})), as we see in the next proposition:
\begin{prop}\label{propconvergencekappa}
Suppose that  $\alpha > 0$, that  the support of $\nu$ is not a discrete subgroup $r
\mathbb{Z}$ for any $r>0$ and that $0<\kappa^{'}(p_{0})<\infty$.
Then
 for every $y>0$, under
  $ \mathbb{P}_{y}^{*}$, $t^{1/\alpha}\chi(t)$ converges in law as $t\rightarrow \infty$ to a random variable $Y$ whose law is specified by

 $$
  \mathbb{E}(k(Y^{\alpha}))=\frac{1}{\alpha m_{1}}E(k(I)I^{-1}),$$  for every measurable function
  $k:\mathbb{R}_{+} \rightarrow  \mathbb{R}_{+}$,
  with $I:=\int_{0}^{\infty}\exp(\alpha \eta_{s})ds$ and $m_{1}:=E(\eta_{1} )=-\kappa^{'}(p_{0})$.

\end{prop}

\begin{proof}

As  $-\kappa^{'}(p_{0})$ is  the mean of the step distribution of
the random walk $S_{n}$ (see Proposition \ref{propSn}),   therefore
$\kappa^{'}(p_{0})>0$ imply that $E(-\eta_{1} )>0$ thus  the
assumption of Theorem 1  in the works of Bertoin and Yor
 \cite{beryor} is fulfilled by  the self-similar Markov process $\chi (t)^{-1}$, which gives the result.
\end{proof}

 We could also  try to use the same method as the one used
in \cite{bergne} for which we need  Proposition 1.7 \cite{ber}. But
in this latter we needed  $\mathbb{E}(\pt x^{p}, X(t)\gd)$ to be
 finite when $p$ is large, and its derivative to be completely monotone. But here neither of these requirements is necessarily true as $\kappa$ is
not necessarily positive when $p$ is large. This explains why we
have to use a different
 method.

\begin{rk}\label{cas=0}
 In the case
 $\kappa^{'}(p_{0})=0$ we can extend this proposition. More precisely if $\int_{\R}\pt x^{p_{0}}|\ln(x)|,\mathbf{s}\gd\nu(d\mathbf{s})<\infty$, $$J:=\int_{1}^{\infty}\frac{x \nu^{-}((x,\infty))dx}{1+\int_{0}^{x}dy\int_{y}^{\infty}\nu^{-}((-\infty,-z))dz}<\infty,$$
(where $\nu^{-}$ is the image of $\widetilde{\nu}$ by the map
$u\rightarrow -u$ and $\widetilde{\nu}$ is defined in Proposition
\ref{propSn}) and $E\left(\log^{+}\int_{0}^{T_{1}}
e^{-\eta_{s}}ds\right)<\infty$ (with $T_{z}:=\inf\{t : -\eta_{t}\geq
z\}$)
 hold then,
 for any $y>0$
  under $ \mathbb{P}_{y}^{*}$, $t^{1/\alpha}\chi(t)$ converge in law as $t\rightarrow
  \infty$, to a random variable $\tilde{Y}$ whose law is specified
  by

 for any bounded and continuous function $k$ and for $t> 0$:
$$\mathbb{E}(k(\tilde{Y}^{\alpha}))=\lim_{\lambda\rightarrow 0}\frac{1}{\lambda}E(I_{\lambda}^{-1}k(I_{\lambda})),$$
  where $I_{\lambda}=\int_{0}^{\infty}\exp(\alpha \eta_{s}-\lambda s)ds$.

The proof is the same as the previous one using Theorem 1 and
Theorem 2 from the works of Caballero and Chaumont \cite{cabcha}
instead of \cite{beryor}.
\end{rk}

\subsection{Convergence of the mean measure and $\mathrm{L}^{p}$-convergence.}
We encode the configuration of masses $X(t)=\{(X_{i}(t))_{1\leq
i\leq \sharp \mathbf{X}(t)} \}$ by the weighted empirical measure
 $$\sigma_{t}:=\sum_{i=1}^{\sharp \mathbf{X}(t)}
X_{i} ^{p_{0}}( t)
\delta_{t^{1/\alpha}X_{i} (t)}
$$
which has total mass $M(t)$.

 The associated mean measure
$\sigma_{t}^{*}$ is defined by the formula
$$\int_{0}^{\infty}k(x)\sigma_{t}^{*} (dx)=\mathbb{E}\left( \int_{0}^{\infty}k (x)\sigma_{t} (dx)\right)$$
which is required to hold for all compactly supported continuous
functions $k$. Since $M(t)$ is a martingale, $\sigma_{t}^{*}$ is a
probability measure. We interest us to the convergence of this
measure. This convergence was already established in the case of
binary conservative fragmentation (see the results of Brennan and
Durrett \cite{bredur1} and \cite{bredur2}). A very useful tool  for
this is the renewal theorem, for which they needed the fact that the
process $\chi (t)$ is decreasing; here we no longer have  such a
monotonicity property. See also Theorem 2 and 5 of \cite{bergne},
Theorem 1.3 of \cite{ber} and Proposition 4 of \cite{kre} for
Theorems about empirical measure for measure which have a
conservative property $\nu (s_{i}\leq 1 \ \forall
i\in\mathbb{N})=1$.

 Nonetheless, with Proposition \ref{propconvergencekappa} and Lemma \ref{lemmeequaliteloikappa}, we easily get:

\begin{cor}\label{coro6}
 With the assumptions of Proposition \ref{propconvergencekappa}   we get:

\begin{enumerate}
\item
The measures $\sigma_{t}^{*}$ converge weakly, as
$t\rightarrow\infty$, to the distribution of $Y$ i.e. for any
continuous bounded function
$k:\mathbb{R}_{+}\rightarrow\mathbb{R}_{+}$ , we have:
   $$
  \mathbb{E}\left(\pt x^{p_{0}}k(t^{1/\alpha} x),X(t)\gd\right)\underset{t\rightarrow\infty}{\rightarrow}  \mathbb{E}(k(Y)).$$ \item For all $p_{+}>p>p_{0}$: $$t^{(p-p_{0})/\alpha}\mathbb{E}\left(\pt x^{p},X(t)\gd\right)\underset{t\rightarrow\infty} {\rightarrow}\mathbb{E}
  (Y^{p-p_{0}}).$$

\end{enumerate}

\end{cor}

We now formulate a more precise result concerning the convergence
of the empirical measure:

\begin{theo}\label{theoprincipale}
Under the same assumptions as in Proposition
\ref{propconvergencekappa}  we get that for every  bounded
continuous function
  $k$:

  $$\mathrm{L}^{p}-\lim_{t\rightarrow\infty}\int_{0}^{\infty}k(x)\sigma_{t}(dx)=M_{\infty} \mathbb{E}(k(Y))= \frac{M_{\infty}}{\alpha m}E(k(I)I^{-1}),$$
  for some $p>1.$

\end{theo}

\begin{rk} A slightly different version of  Corollary \ref{coro6} and Theorem \ref{theoprincipale} exists also under the assumptions  in Remark  \ref{cas=0}.
\end{rk}

See also Asmussen and Kaplan \cite{asmkap} and \cite{asmkap2} for
a closely related result.
\begin{proof}

We follow the same method as Section 1.4.  in \cite{ber} and in this
direction we use Lemma 1.5 there: for $(\lambda(t))_{t\geq
0}=(\lambda_{i}(t),i\in\mathbb{N})_{t\geq 0}$ a sequence of
non-negative random variables such that for fixed $p>1$
$$\sup_{t\geq 0}\mathbb{E}\left(\left( \sum_{i=1}^{\infty} \lambda_{i}(t)\right)^{p}\right)<\infty\ \ \ \text{and}\ \ \lim_{t \rightarrow \infty}\mathbb{E}\left( \sum_{i=1}^{\infty} \lambda_{i}(t)\right)=0,$$ and for $(Y_{i}(t), i\in\mathbb{N})$ a sequence of random variables which are independent conditionally on $\lambda(t)$, we assume that there exists a sequence $(\overset{-}{Y_{i}},i\in\mathbb{N})$ of i.i.d variables in $\mathrm{L}^{p}(\mathbb{P}),$ which is independent  of $\lambda(t)$  for each fixed $t$, and such that $|Y_{i}(t)
|\leq \overset{-}{Y}_{i}$ for all $i\in\mathbb{N}$ and $t\geq 0$.

\medskip

Then we know from Lemma 1.5 in \cite{ber} that
\begin{equation}
\lim_{t \rightarrow \infty} \sum_{i=1}^{\infty} \lambda_{i}(t)(Y_{i}(t)-\mathbb{E}\left(Y_{i}(t)|\lambda(t))\right)=0 .   \label{lemma15}  \end{equation}

Now, let $k$ be a continuous function bounded by 1 and let
$$A_{t}:=\pt x^{p_{0}}k(t^{1/\alpha} x) ,X(t)\gd . $$

By application of the Markov property at time $t$ for $A_{t+s}$ and
the self-similarity property of the process $\mathbf{X}$ we can
rewrite $A_{t+s}$ as

$$\sum_{i=1}^{\sharp \mathbf{X}(t)}\lambda_{i}(t)Y_{i}(t,s)$$
where
 $\lambda_{i}(t):=X_{i}^{p_{0}}(t)$ and
$$Y_{i}(t,s):= \pt x^{p_{0}} k((t+s)^{1/\alpha}X_{i}(t)x), \mathbf{X}_{i,.}(s)\gd,$$
with $\mathbf{X}_{1,.}$, $\mathbf{X}_{2,.}$, ... a sequence of
i.i.d. copies of $\mathbf{X}$ which is independent of
$\mathbf{X}(t)$.

By Theorem \ref{theoremdecvmartingale} we get that
$$\sup_{t\geq 0} \mathbb{E}\left(\left(\sum_{i=1}^{\sharp \mathbf{X}(t)}\lambda_{i}(t)\right)^{p}\right)<\infty.$$
By the last corollary we also obtain that
$$ \mathbb{E}\left(\sum_{i=1}^{\sharp \mathbf{X}(t)}\lambda_{i}^{p}(t)\right)\sim t^{-(p-1)p_{0}}\mathbb{E}
(\chi^{(p-1)p_{0}} (1))\rightarrow 0,$$ as $t\rightarrow \infty$.

\noindent  Moreover  the variables $Y_{i}(t,s)$ are uniformly bounded by
$$Y_{i}= \sup_{s\geq 0}\pt x^{p_{0}},\mathbf{X}_{i,.}(s)\gd,$$
 which are i.i.d. variables and also bounded in $\mathrm{L}^{p}(\mathbb{P})$ thanks to Doob's inequality (as  $\pt x_{p_{0}} , \mathbf{X}_{i,.}(s) \gd$ is a martingale bounded in  $\mathrm{L}^{p}(\mathbb{P})$).

 Thus we may apply (\ref{lemma15}), which reduces the study to that of the asymptotic behavior of:
 $$\sum_{i=1}^{\sharp \mathbf{X}(t)}\lambda_{i}(t)\mathbb{E}(Y_{i}(t,s)| \mathbf{X}(t)),$$
as $ t$ tends to $\infty$. On the event $\{X_{i}(t)=y\}$, we get
$$\mathbb{E}(Y_{i}(t,s)|\mathbf{X}(t))=\mathbb{E}\left(\pt x^{p_{0}}k((t+s)^{1/\alpha }yx), \mathbf{X}(s)\gd \right).$$
Then by Lemma \ref{lemmeequaliteloikappa}:
$$\mathbb{E}\left(\pt x^{p_{0}}k((t+s)^{1/\alpha }yx),\mathbf{X}_{i,.}(s)\gd \right)=\mathbb{E}^{*}\left(k\left((t+s)^{1/\alpha}y\chi(s)\right)\right).$$
With Proposition \ref{propconvergencekappa}, we obtain
$$\lim_{t\rightarrow\infty}\mathbb{E}^{*}\left(k\left((t+s)^{1/\alpha}y\chi(s)\right)\right)=\mathbb{E}\left(k\left(Y\right)\right)
.$$ Moreover recall from Corollary \ref{convergencemartingale} that
$\sum_{i=1}^{\sharp \mathbf{X}(t)}\lambda_{i}(t)$  converges to
$M_{\infty}$  in $\mathrm{L}^{p}(\mathbb{P}).$ Therefore we finally
get that when $t$ goes to infinity:

 $$\sum_{i=1}^{\sharp \mathbf{X}(t)}\lambda_{i}(t)\mathbb{E}(Y_{i}(t,s)| X(t))\sim\mathbb{E}\left(k\left(Y\right)\right)\sum_{i=1}^{\sharp \mathbf{X}(t)}\lambda_{i}(t)\sim\mathbb{E}
\left( k \left(Y\right)\right)M_{\infty}.$$

\end{proof}

\begin{appendix}
\section{Further results about the intrinsic process}
We will give more general properties about the intrinsic process
$\{M_{Q},\ Q\subset \mathcal{U}\},$ $M_{Q}=\sum_{u\in
M}\xi_{u}^{p_{0}}.$ For a line $Q$, $\{M_{Q}\}$ is adapted to the
filtration $\{\mathcal{H}_{L}\}$. We use the abuse of notation that
$M_{n}$ stand for the process $M_{L_{n}}$, with
$L_{n}=\{u\in\mathcal{U}:\ |u|=n\}$ the labels of the $n$-th
generation. We introduce new definitions, we say that a line $Q$
\textit{covers} $L$, if $Q\succeq L$ and any individual stemming
from $L$ either stems from $Q$ or has progeny in $Q$. If $Q$ covers
the ancestor it may simply be called \textit{covering}. Let $
\mathcal{ C}_{0}$ be the class of covering lines with finite maximal
generation. We denoted the generation of $Q$: $|Q|=\sup_{u\in
Q}|u|$. The origin of the intrinsic martingale comes from real time
martingale of Nerman \cite{ner}.

Also for $r\in \mathbb{R}_{+}^{*}$, let $\vartheta_{r}$ be the
structural measure:
$$\vartheta_{r}(B):=\mathbb{E}_{r}(\sharp \{u\in\mathcal{U} : \xi_{u}\in B\})=\sum_{i=1}^{\infty}\nu(rs_{i}\in
B) \ \ \ \text{for} \ B\subset \mathcal{B},$$ where $\mathcal{B}$ is
the Borel algebra on  $\Rpe$. Let the reproduction measure $ \mu$ on
the sigma-field $\mathcal{B}\otimes\mathcal{B}$ be such that for
every $r\geq 0$:$$\mu (r, dv\times du):=r^{\alpha}
e^{-r^{\alpha}u}du\sum_{i=1}^{\infty}\nu (rs_{i}\in dv)$$ and for
any $\lambda\in\mathbb{R}$
$$\mu_{\lambda} (r,
dv\times du):=e^{-\lambda u}\mu (r, du\times dv).$$ The composition
operation $\ast $ denotes the Markov transition on the size space
$\mathbb{R}_{+}$ and convolution on the time space $\Rp$, so that:
for all $A\in\mathcal{B}$ and $B\in\mathcal{B}$,
$$\mu^{\ast 2}(s, A\times B)=\mu \ast \mu (s, A\times B)=\int_{\Rp \times\Rp} \mu (r,
A\times (B-u))\mu (s, dr\times du).$$ With the convention that the
$\ast$-power 0 is $\nbOne_{\{A\times B\}}(s,0)$ which gives all the
mass to $(s,0)$. We define the renewal measure as
$$\psi_{\lambda}:=\sum_{0}^{\infty}\mu_{\lambda}^{\ast n}.$$
Let $$\alpha^{'}:=\inf\{\lambda\ :\ \ \psi_{\lambda}(r,\Rp\times\Rp
)<\infty \ \ \text{for some} \ \ r\in \Rp \}.$$ Moreover as
\begin{equation*}
    \mu_{\lambda}(r,\Rp\times\Rp )=
    \begin{cases} mr^{\alpha }/(r^{\alpha} +\lambda )&\hbox{if }\ \lambda >-r^{\alpha}\\
    \infty &\hbox{ else}, \end{cases}
      \end{equation*}
thus \begin{equation*}
    \psi_{\lambda}(r,\Rp\times\Rp )<\infty\ \
   \hbox{if and only if}\ \lambda <(r/(m-1))^{1/\alpha} \end{equation*}
   therefore   we get $\alpha^{'}=0$. For $A\in\mathcal{B }$, let \begin{equation}\pi
(A):=\lim_{n\rightarrow \infty}\mu^{\ast n}(1, A\times \Rp
)\label{pi}\end{equation} which is well defined as $\mu^{\ast n}(1,
A\times \Rp )$ is a decreasing function in $n$ and nonnegative. Let
$h(s):=s^{p_{0}}$ for all $s\in\Rp$ and $\beta:=1$. These objects
correspond to those defined in \cite{jag}.

Recall that the Galton-Watson process $(Z_{n}, n\geq 0))$ is equal
to $(\sharp \{ u\in\mathcal{U}: \ \xi_{u}>0 \ \text{and}\ \
|u|=n\},n \geq 0)$.

\textbf{We suppose that} $$m:=\mathbb{E}(Z_{1})<\infty,$$ i.e.
$\int_{\R} \sharp s\nu (d\mathbf{s})<\infty$ this assumption is
slightly stronger than (\ref{hyposurmesurededis}), therefore we get
that:

\begin{prop}
\begin{enumerate}
\item If $L\preceq Q$ are lines, then
$$\mathbb{E}(M_{Q}| \mathcal{H}_{L})\leq M_{L}.$$

If $Q$ verifies $|Q|<\infty $ and covers $L$, then
$$\mathbb{E}(M_{Q}| \mathcal{H}_{L})= M_{L}.$$

\item For all  $s>0$, $\{M_{L};\ L\in \mathcal{C}_{0}\}$ is
uniformly $\mathbb{P}_{s}$-integrable.
\item There is a random variable $M \geq 0$ such that for
$\pi$-almost all $s>0$
$$M_{L}= \mathbb{E}_{s}(M| \mathcal{H}_{L})$$ and
$M_{L}\overset{L^{1}(\mathbb{P}_{s})}{\rightarrow} M,$ as $L\in
\mathcal{C}_{0}$ filters $(\preceq )$.  If $\varsigma_{n}\preceq
\varsigma_{n+1}\in \mathcal{C}_{0}$ and to any $x\in \mathcal{U}$
there is an $\varsigma_{n}$ such that $x$ has progeny in
$\varsigma_{n}$, $M_{\varsigma_{n}}\rightarrow M$, as $n\rightarrow
\infty$, also a.s. $\mathbb{P}_{s}$.
\end{enumerate}

\end{prop}

A consequence of the first and second points  applied for
$L_{n}=\{u\in\mathcal{U}:\ |u|=n\}$ and $L_{m}=\{u\in\mathcal{U}:\
|u|=m\}$ with $m\geq n\geq 0$, is that $M_{n}$ is a martingale and
the uniform $\mathbb{P}_{s}$-integrability of this martingale. The
third point applied for the  lines $\tau_{t}$ give the convergence
of $M(t)$ in $\mathrm{L}^{1}(\mathbb{P}_{s})$ and almost surely.

\begin{proof}

$\bullet$ First the conditions of Malthusian population are
fulfilled, thus by Theorem 5.1 in \cite{jag} we get the first point.

Let
$\overline{\xi}:=\int_{\Rp\times\Rp}h(s)r^{\alpha}e^{-tr^{\alpha}}dt
\vartheta_{1}(ds)=\sum_{|u|=1}\xi_{u}^{p_{0}}$ and
$\mathbb{E}_{\pi}$ be the expectation with respect to
$\int_{\Rp}\mathbb{P}_{s}(dw)\pi (ds)$. Therefore,
$$\mathbb{E}_{\pi}(\overline{\xi}\log^{+}\overline{\xi})=\int_{\Rp}\mathbb{E}_{x}
\left(\sum_{i=1}^{\infty}
\xi_{i}^{p_{0}}\left(\log^{+}\sum_{j=1}^{\infty}
\xi_{j}^{p_{0}}\right)\right)\pi (dx),$$ and it follows readily from
the Malthusian hypotheses and the fact that
$\sum_{|u|=n}\xi_{u}^{pp_{0}}$ is a supermartingale, that this
quantity is finite. Therefore the assumption of  Theorem 6.1 of
\cite{jag} are check, which gives by Theorem 6.1 of  \cite{jag} the
second point and by Theorem 6.3 of  \cite{jag} we get the third
point.
\end{proof}
\end{appendix}

\bigskip
\noindent{\bf Acknowledgements: } I wish to thank J. Bertoin for his
help and suggestions. I also wish to thank the anonymous referees of
an  earlier draft for their detailed comments and suggestions.


\bigskip

\begin{thebibliography}{99}

\bibitem{asmkap} S. \textsc{Asmussen} and N. \textsc{Kaplan} (1976). Branching random walks. I. \textit{Stochastic Process. Appl.}  \textbf{4} , no. 1, 1-13.

\bibitem{asmkap2} S. \textsc{Asmussen} and N. \textsc{Kaplan} (1976).  Branching random walks. II. \textit{Stochastic Process. Appl.}  \textbf{4} , no. 1, 15-31.

\bibitem{athney} K. B. \textsc{Athreya} and P. E. \textsc{Ney} (1972).
 \textit{Branching processes}.
Springer-Verlag Berlin Heidelberg.


\bibitem{ber} J. \textsc{Bertoin} (2006).
 \textit{Random fragmentation and coagulation processes}.
 Cambridge Univ. Pr.

\bibitem{ber2006} J. \textsc{Bertoin} (2006).
 Different aspects of a random fragmentation model.
\textit{Stochastic Process. Appl.} \textbf{116}, 345-369.

\bibitem{bergne} J. \textsc{Bertoin} and A. V. \textsc{Gnedin} (2004)
Asymptotic laws for nonconservative self-similar fragmentations.
\textit{ Electron. J. Probab.} \textbf{9} , No. 19, 575-593


\bibitem{beryor} J. \textsc{Bertoin} and M. \textsc{Yor} (2002).
 The entrance laws of self-similar Markov processes and exponential functionals of L\'evy processes. \textit{Potential Analysis}
\textbf{17} 389-400.

\bibitem{big} J. D. \textsc{Biggins} (1992).
Uniform convergence of martingales in the branching random walk.
\textit{ Ann. Probab.} \textbf{20}, No. 1, 131-151.


\bibitem{bredur1}  M. D. \textsc{Brennan} and R. \textsc{Durrett} (1986). Splitting intervals.  \textit{Ann.
Probab.}  \textbf{14}  ,  No. 3, 1024-1036.

\bibitem{bredur2}  M. D. \textsc{Brennan} and R. \textsc{Durrett} (1987). Splitting intervals. II. Limit
laws for lengths.  \textit{Probab. Theory Related Fields}
\textbf{75}  No. 1, 109-127.



\bibitem{cabcha} M.E. \textsc{Caballero} and L. \textsc{Chaumont} (2006).
Weak convergence of positive self-similar Markov processes and overshoots of L\' evy processes.
\textit{ Ann. Probab.} \textbf{34}, No. 3, 1012-1034.


\bibitem{cha0} B. \textsc{Chauvin} (1986).
Arbres et processus de Bellman-Harris. \textit{ Ann. Inst. Henri
Poincar\'e.} \textbf{22}, No. 2, 209-232.

\bibitem{cha} B. \textsc{Chauvin} (1991).
Product martingales and stopping lines for branching Brownian motion.
\textit{ Ann. Probab.} \textbf{19}, No. 3, 1195-1205.


\bibitem{har} T. E. \textsc{Harris} (1963).
 \textit{The theory of branching processes. }
Springer



\bibitem{jag} P. \textsc{Jagers} (1989). General branching
processes as Markov fields. \textit{Stochastic Process. Appl.}
\textbf{32}, 183-212.


\bibitem{kin} J. F. C. \textsc{Kingman} (1993).  \textit{Poisson processes.} Oxford Studies in Probability,
3. Oxford Science Publications. The Clarendon Press, Oxford
University Press.

\bibitem{kre} N. \textsc{Krell} (2008). Multifractal spectra and precise rates of decay in homogeneous
fragmentations. \textit{To appear in Stochastic Process. Appl.}

\bibitem{kyp1999} A. E. \textsc{Kyprianou} (1999). A note on branching Lévy processes.  \textit{Stochastic
Process. Appl.}  \textbf{82}  ,  No. 1, 1-14.

\bibitem{kyp2000} A. E. \textsc{Kyprianou} (2000).  Martingale convergence and the stopped branching
random walk.  \textit{Probab. Theory Related Fields}  \textbf{116} ,
no. 3, 405-419.

\bibitem{ner} O. \textsc{Nerman} (1984). \textit{The growth and composition of supercritical branching populations
on general type spaces}. Technical report, Dept. Mathematics,
Chalmers Univ. Technology and Goteborg Univ.


\bibitem{nev} J. \textsc{Neveu} (1986). Arbres et processus de Galton-Watson.  \textit{Ann. Inst. H. Poincaré Probab. Statist.}  \textbf{22},  No. 2, 199-207.


\end{thebibliography}
 \end{document}